\documentclass[english,american]{amsart}
\usepackage[T1]{fontenc}
\usepackage[latin9]{inputenc}
\usepackage[a4paper]{geometry}
\usepackage{amsthm}
\usepackage{amstext}
\usepackage{amssymb}
\usepackage{setspace}
\makeatletter
\numberwithin{equation}{section}
\numberwithin{figure}{section}
\theoremstyle{plain}
\newtheorem{thm}{\protect\theoremname}[section]
  \theoremstyle{definition}
  \newtheorem{example}[thm]{\protect\examplename}
  \theoremstyle{definition}
  \newtheorem{defn}[thm]{\protect\definitionname}
  \theoremstyle{remark}
  \newtheorem{rem}[thm]{\protect\remarkname}
  \theoremstyle{plain}
  \newtheorem{prop}[thm]{\protect\propositionname}
  \theoremstyle{plain}
  \newtheorem{cor}[thm]{\protect\corollaryname}
  \theoremstyle{remark}
  \newtheorem*{acknowledgement*}{\protect\acknowledgementname}
  \theoremstyle{plain}
  \newtheorem{lem}[thm]{\protect\lemmaname}

\usepackage{amscd}
\usepackage{mathptmx}
\usepackage{bbm}

\newcommand{\e}{\mathrm e}
\newcommand{\1}{\mathbbm{1}}
\newcommand{\N}{\mathbb{N}}

\newcommand{\Z}{\mathbb{Z}}

\newcommand{\R}{\mathbb{R}}

\renewcommand{\Pi}{\pi}

\renewcommand{\hat}{\widehat}

\newcommand\id{\mathrm{id}}

\DeclareMathOperator*{\A}{\mathit{A}}

\usepackage{hyperref}
\@ifundefined{textmu}
 {\usepackage{textcomp}}{}

\makeatother

\usepackage{babel}
  \addto\captionsamerican{\renewcommand{\acknowledgementname}{Acknowledgement}}
  \addto\captionsamerican{\renewcommand{\corollaryname}{Corollary}}
  \addto\captionsamerican{\renewcommand{\definitionname}{Definition}}
  \addto\captionsamerican{\renewcommand{\examplename}{Example}}
  \addto\captionsamerican{\renewcommand{\lemmaname}{Lemma}}
  \addto\captionsamerican{\renewcommand{\propositionname}{Proposition}}
  \addto\captionsamerican{\renewcommand{\remarkname}{Remark}}
  \addto\captionsamerican{\renewcommand{\theoremname}{Theorem}}
  \addto\captionsenglish{\renewcommand{\acknowledgementname}{Acknowledgement}}
  \addto\captionsenglish{\renewcommand{\corollaryname}{Corollary}}
  \addto\captionsenglish{\renewcommand{\definitionname}{Definition}}
  \addto\captionsenglish{\renewcommand{\examplename}{Example}}
  \addto\captionsenglish{\renewcommand{\lemmaname}{Lemma}}
  \addto\captionsenglish{\renewcommand{\propositionname}{Proposition}}
  \addto\captionsenglish{\renewcommand{\remarkname}{Remark}}
  \addto\captionsenglish{\renewcommand{\theoremname}{Theorem}}
  \providecommand{\acknowledgementname}{Acknowledgement}
  \providecommand{\corollaryname}{Corollary}
  \providecommand{\definitionname}{Definition}
  \providecommand{\examplename}{Example}
  \providecommand{\lemmaname}{Lemma}
  \providecommand{\propositionname}{Proposition}
  \providecommand{\remarkname}{Remark}
\providecommand{\theoremname}{Theorem}

\begin{document}

\title{Group Extended Markov Systems, Amenability, and the Perron-Frobenius
Operator}

\author{Johannes Jaerisch}

\thanks{The author was supported by the research fellowship JA 2145/1-1 of
the German Research Foundation (DFG)}

\address{Department of Mathematics, Graduate School of Science Osaka University,
1-1 Machikaneyama Toyonaka, Osaka, 560-0043 Japan }

\email{jaerisch@cr.math.sci.osaka-u.ac.jp}

\subjclass[2000]{Primary 37D35, 43A07, 37C30; Secondary 37C85. }

\keywords{Thermodynamic formalism; group extension; amenability; Perron-Frobenius
operator; \foreignlanguage{english}{Gurevi\v c pressure.}}
\begin{abstract}
We characterise amenability of a countable group in terms of the spectral
radius of the Perron-Frobenius operator associated to a group extension
of a countable Markov shift and a H\"older continuous potential.
This extends a result of Day for random walks and recent work of Stadlbauer
for dynamical systems. Moreover, we show that, if the potential satisfies
a symmetry condition with respect to the group extension, then the
logarithm of the spectral radius of the Perron-Frobenius operator
is given by the \foreignlanguage{english}{Gurevi\v c} pressure of
the potential. 
\end{abstract}
\maketitle

\section{Introduction and Statement of Results}

Kesten (\cite{MR0112053}) characterised amenability of a countable
group $G$ by  the growth of the return probability of a symmetric
random walk. Day (\cite{MR0159230}) gave a criterion for amenability
in terms of the spectrum of a convolution operator acting on the Banach
space $\ell^{p}\left(G\right)$, $p\in\N$, without assuming the random
walk to be symmetric. Recently, the relationship between amenability
and dynamical systems was studied in the framework of the thermodynamic
formalism (\cite{MR2322540,Jaerisch11a,Stadlbauer11}). Stadlbauer
(\cite{Stadlbauer11}) used the \foreignlanguage{english}{Gurevi\v c}
pressure of a H\"older continuous potential (\cite{MR1738951}) to
give an extension of Kesten's criterion for amenability to group extensions
of Markov shifts. To characterise amenability via the \foreignlanguage{english}{Gurevi\v c}
pressure, it is necessary to impose certain symmetry assumptions on
the potential and on the Markov shift. 

In this paper, we characterise amenability in terms of the spectral
radius of the Perron-Frobenius operator associated to a group extension
of a countable Markov shift and a H\"older continuous potential,
where neither the group extension nor the potential is assumed to
be symmetric. Another main result is to relate the spectral radius
of the Perron-Frobenius operator to the \foreignlanguage{english}{Gurevi\v c}
pressure.

Throughout this paper, let $\Sigma$ denote a Markov shift with countable
alphabet $I$ and left shift map $\sigma:\Sigma\rightarrow\Sigma$
(see Section 2). For a countable group $G$ and a semigroup homomorphism
$\Psi:I^{*}\rightarrow G$, where $I^{*}$ denotes the free semigroup
generated by $I$, the \emph{group extended Markov system} $\left(\Sigma\times G,\sigma\rtimes\Psi\right)$
is given by 
\[
\sigma\rtimes\Psi:\Sigma\times G\rightarrow\Sigma\times G,\,\,\left(\sigma\rtimes\Psi\right)\left(\omega,g\right):=\left(\sigma\left(\omega\right),g\Psi\left(\omega_{1}\right)\right),\,\,\left(\omega,g\right)\in\Sigma\times G.
\]
Let $\pi_{1}:\Sigma\times G\rightarrow\Sigma$ denote the canonical
projection. If $\Sigma$ is topologically mixing and satisfies the
big images and preimages (b.i.p.) property (see Definition \ref{mixing-definitions})
and if $\varphi:\Sigma\rightarrow\R$ is Hölder continuous with finite
\foreignlanguage{english}{Gurevi\v c} pressure $\mathcal{P}\left(\varphi,\sigma\right)$
 (see Definition \ref{def:gurevich}), then the Perron-Frobenius operator
$\mathcal{L}_{\varphi\circ\pi_{1}}$, given by $\mathcal{L}_{\varphi\circ\pi_{1}}\left(f\right)\left(x\right):=\sum_{\left(\sigma\rtimes\Psi\right)\left(y\right)=x}\e^{\varphi\circ\pi_{1}\left(y\right)}f\left(y\right)$,
acts as a bounded linear operator on a certain Banach space $\left(\mathcal{H}_{\infty},\vert\cdot\vert_{\infty}\right)$
of functions $f:\Sigma\times G\rightarrow\R$ (\cite{Stadlbauer11}).
We refer to Section 3 for the definition of $\left(\mathcal{H}_{\infty},\vert\cdot\vert_{\infty}\right)$
and denote by $\rho(\mathcal{L}_{\varphi\circ\pi_{1}})$ the spectral
radius of $\mathcal{L}_{\varphi\circ\pi_{1}}:\mathcal{H}_{\infty}\rightarrow\mathcal{H}_{\infty}$.

Our main result is the following characterisation of amenability for
group extended Markov systems. 
\begin{thm}
\label{thm:amenable-spectralradius-equivalence}Let $\Sigma$ be a
topologically mixing Markov shift with the b.i.p. property and let
$\left(\Sigma\times G,\sigma\rtimes\Psi\right)$ be an irreducible
group extended Markov system. Let $\varphi:\Sigma\rightarrow\R$ be
Hölder continuous with \foreignlanguage{english}{\textup{$\mathcal{P}\left(\varphi,\sigma\right)<\infty$}}.
Then we have $\log\rho\left(\mathcal{L}_{\varphi\circ\pi_{1}}\right)=\mathcal{P}\left(\varphi,\sigma\right)$
if and only if $G$ is amenable. In general, we have $\log\rho\left(\mathcal{L}_{\varphi\circ\pi_{1}}\right)\le\mathcal{P}\left(\varphi,\sigma\right)$.
\end{thm}
Theorem \ref{thm:amenable-spectralradius-equivalence} provides an
extension of a result of Day for random walks on groups (\cite[Theorem 1]{MR0159230}).
The Perron-Frobenius operator corresponds to the convolution operator
acting on $\ell^{2}\left(G\right)$ in Day's setting. Theorem \ref{thm:amenable-spectralradius-equivalence}
also generalises \cite[Theorem 3.21]{Jaerisch11a}, where locally
constant potentials were considered (see Remark \ref{tams-comments}
for details). 

Stadlbauer (\cite{Stadlbauer11}) gave a criterion for amenability
in terms of the \foreignlanguage{english}{Gurevi\v c} pressure\foreignlanguage{english}{,
which is more in the spirit of }Kesten's characterisation of amenability
in terms of the growth rate of return probabilities (see \cite{MR0109367,MR0112053}).
More precisely, Stadlbauer proved the following two implications.
Here, $\mathcal{P}\left(\varphi\circ\pi_{1},\sigma\rtimes\Psi\right)$
refers to the \foreignlanguage{english}{Gurevi\v c} pressure of $\varphi\circ\pi_{1}$
with respect to $\sigma\rtimes\Psi$. 
\begin{enumerate}
\item \label{enu:stadlbauer-fullpressure-amenable} Under the assumptions
of Theorem \ref{thm:amenable-spectralradius-equivalence}, if $\mathcal{P}\left(\varphi\circ\pi_{1},\sigma\rtimes\Psi\right)=\mathcal{P}\left(\varphi,\sigma\right)$,
then $G$ is amenable (\cite[Theorem 5.4]{Stadlbauer11}).
\item \label{enu:stadlbauer-amenablesymmetric-fullpressure}Let $\Sigma$
be topologically mixing and let $\left(\Sigma\times G,\sigma\rtimes\Psi\right)$
be a symmetric group extension. If $G$ is amenable and if $\varphi$
is weakly symmetric with $\mathcal{P}\left(\varphi,\sigma\right)<\infty$,
then $\mathcal{P}\left(\varphi\circ\pi_{1},\sigma\rtimes\Psi\right)=\mathcal{P}\left(\varphi,\sigma\right)$
(\cite[Theorem 4.1]{Stadlbauer11}).
\end{enumerate}
Comparing Stadlbauer's results with Theorem \ref{thm:amenable-spectralradius-equivalence},
we note that the spectral radius of the Perron-Frobenius operator
can be used to characterise amenability for an arbitrary Hölder continuous
potential. In contrast to this, the criterion in terms of the \foreignlanguage{english}{Gurevi\v c}
pressure\foreignlanguage{english}{ involves a certain symmetry condition
on the} Markov shift and the potential under consideration (see \cite[p.455]{Stadlbauer11}
for the definition of a symmetric group extension and a weakly symmetric
potential). For an amenable group, it may happen that $\mathcal{P}\left(\varphi\circ\pi_{1},\sigma\rtimes\Psi\right)<\mathcal{P}\left(\varphi,\sigma\right)$,
if $\varphi$ is not weakly symmetric. In fact, the gap between these
pressures can be arbitrarily large, as the following example illustrates.
\begin{example}
Let $I:=\left\{ \pm1\right\} $ and consider the full shift $\Sigma:=I^{\N}$
and the group extended Markov system given by the canonical semigroup
homomorphism $\Psi:I^{*}\rightarrow\left(\Z,+\right)$. For a real
number $\lambda>0$, let $\varphi_{\lambda}:\Sigma\rightarrow\R$
be given by $\varphi_{\lambda}\left(x\right):=\lambda x_{1}$. Then
we have $\mathcal{P}\left(\varphi_{\lambda}\circ\pi_{1},\sigma\rtimes\Psi\right)=\log(2)<\log(\e^{\lambda}+\e^{-\lambda})=\mathcal{P}\left(\varphi_{\lambda},\sigma\right)$. 
\end{example}
For a Markov shift with a finite alphabet, it is well-known that the
\foreignlanguage{english}{Gurevi\v c} pressure of a Hölder continuous
potential coincides with the logarithm of the spectral radius of the
Perron-Frobenius operator. However, if the Markov shift is constructed
over an infinite alphabet, then the \foreignlanguage{english}{Gurevi\v c}
pressure is less or equal to the logarithm of the spectral radius.
The \foreignlanguage{english}{Gurevi\v c} pressure describes the
growth of iterates of the Perron-Frobenius operator on functions supported
on a cylindrical set, whereas for the spectral radius, we have to
consider functions supported on the whole space (see Lemma \ref{lem:growth-is-spectralradius}).
To relate the spectral radius of the Perron-Frobenius operator to
the \foreignlanguage{english}{Gurevi\v c} pressure, we\foreignlanguage{english}{
introduce the following notions of symmetry, }generalizing\foreignlanguage{english}{
those given in }\cite[Definition 3.10]{Jaerisch11a}.
\begin{defn}
Let $\left(\Sigma\times G,\sigma\rtimes\Psi\right)$ be a group extended
Markov system and $\varphi:\Sigma\rightarrow\R$. Let $\alpha\ge1$.
We say that $\varphi$ is \emph{asymptotically $\alpha$-symmetric
with respect to $\Psi$, }if there exist $n_{0}\in\N$ and sequences
$\left(c_{n}\right)\in\left(\R^{+}\right)^{\N}$ and $\left(N_{n}\right)\in\N^{\N}$
with the property that $\lim_{n}\left(c_{n}\right)^{1/(2n)}=\alpha$,
$\lim_{n}n^{-1}N_{n}=0$ and such that, for each $g\in G$ and for
all $n\ge n_{0}$, we have 
\begin{equation}
\sum_{\omega\in\Sigma^{n}:\Psi\left(\omega\right)=g}\e^{\sup S_{n}\varphi_{|\left[\omega\right]}}\le c_{n}\sum_{\omega\in\Sigma^{*}:\Psi\left(\omega\right)=g^{-1},\, n-N_{n}\le\left|\omega\right|\le n+N_{n}}\e^{\sup S_{\left|\omega\right|}\varphi_{|\left[\omega\right]}}.\label{eq:asympt-symmetric}
\end{equation}
We say that $\varphi$ is \emph{compactly asymptotically $\alpha$-symmetric
with respect to $\Psi$ }if there exists a sequence $\left(\Sigma_{k}\right)_{k\in\N}$
of topologically mixing subshifts of $\Sigma$ with finite alphabet
$I_{k}\subset I$, such that $\bigcup_{k\in\N}I_{k}=I$, the set $\Psi\left(\Sigma_{k}^{*}\right)$
is a subgroup of $G$, and $\varphi_{|\Sigma_{k}}$ is asymptotically
$\alpha$-symmetric with respect to $\Psi_{|I_{k}^{*}}$. If $\alpha$
is not specified, then we will tacitly assume that $\alpha=1$. \end{defn}
\begin{rem}
\label{rem:symmetry-coboundary}If $\varphi$ is (compactly) asymptotically
$\alpha$-symmetric with respect to $\Psi$, then so is $\varphi+\log h-\log h\circ\sigma+P$,
for each $P\in\R$ and for each function $h:\Sigma\rightarrow\R^{+}$,
which is bounded away from zero and infinity. 
\end{rem}
The following proposition generalises \cite[Corollary 3.17]{Jaerisch11a},
where a locally constant potential $\varphi$ on a finite state Markov
shift was considered (see Remark \ref{tams-comments} for details).
The proposition relates Theorem \ref{thm:amenable-spectralradius-equivalence}
to Stadlbauer's results and is also of independent interest, since
it relates the \foreignlanguage{english}{Gurevi\v c} pressure of
a not necessarily recurrent potential (\cite{MR1818392}) on $\Sigma\times G$
to the spectrum of the Perron-Frobenius operator. 
\begin{prop}
\label{prop:symmetry-pressureisspectralradius}Under the assumptions
of Theorem \ref{thm:amenable-spectralradius-equivalence}, we have
$\mathcal{P}\left(\varphi\circ\pi_{1},\sigma\rtimes\Psi\right)\le\log\rho\left(\mathcal{L}_{\varphi\circ\pi_{1}}\right)$.
If additionally $\varphi$ is asymptotically $\alpha$-symmetric with
respect to $\Psi$, for some $\alpha\ge1$, then 
\[
\mathcal{P}\left(\varphi\circ\pi_{1},\sigma\rtimes\Psi\right)\ge\log\rho\left(\mathcal{L}_{\varphi\circ\pi_{1}}\right)-\log\alpha.
\]

\end{prop}
By combining Theorem \ref{thm:amenable-spectralradius-equivalence}
and the first assertion in Proposition \ref{prop:symmetry-pressureisspectralradius},
we obtain Stadlbauer's result (\ref{enu:stadlbauer-fullpressure-amenable}).
Theorem \ref{thm:amenable-spectralradius-equivalence} and the second
assertion in Proposition \ref{prop:symmetry-pressureisspectralradius}
give the following implication, which is similar to Stadlbauer's result
(\ref{enu:stadlbauer-amenablesymmetric-fullpressure}) and \cite[Theorem 5.3.11]{JaerischDissertation11}.
\begin{cor}
\label{cor:amenable-fullpressure-bip}Under the assumptions of Theorem
\ref{thm:amenable-spectralradius-equivalence}, if $G$ is amenable
and if $\varphi$ is asymptotically $\alpha$-symmetric with respect
to $\Psi$, for some $\alpha\ge1$, then $\mathcal{P}\left(\varphi\circ\pi_{1},\sigma\rtimes\Psi\right)\ge\mathcal{P}\left(\varphi,\sigma\right)-\log\alpha$.
\end{cor}
If $\varphi$ is not assumed to be Hölder continuous, or if $\Sigma$
does not satisfy the b.i.p. property as in (\ref{enu:stadlbauer-amenablesymmetric-fullpressure}),
then a Gibbs measure for $\varphi$ does not exist. To obtain an inequality
as in the previous corollary, the existence of an approximating sequence
of Hölder continuous potentials is sufficient. More precisely, suppose
there exists a sequence of subgroups $\left(G_{j}\right)$ of $G$,
group extended Markov systems $\left(\Sigma_{j}\times G_{j},\sigma_{j}\rtimes\Psi_{j}\right)$
and potentials $\varphi_{j}:\Sigma_{j}\rightarrow\R$, such that $\varphi_{j}$
is asymptotically $\alpha_{j}$-symmetric with respect to $\Psi_{j}$
and the assumptions of Theorem \ref{thm:amenable-spectralradius-equivalence}
hold for each $j\in\N$. If moreover $\lim_{j}\mathcal{P}\left(\varphi_{j},\sigma_{j}\right)\ge\mathcal{P}\left(\varphi,\sigma\right)$,
$\lim_{j}\mathcal{P}\left(\varphi_{j}\circ\pi_{1},\sigma_{j}\rtimes\Psi_{j}\right)\le\mathcal{P}\left(\varphi\circ\pi_{1},\sigma\rtimes\Psi\right)$
and $\lim_{j}\alpha_{j}\le\alpha$, then we have \foreignlanguage{english}{$\mathcal{P}\left(\varphi\circ\pi_{1},\sigma\rtimes\Psi\right)\ge\mathcal{P}\left(\varphi,\sigma\right)-\log\alpha$,
provided that $G$ is amenable.  Using this approach, Stadlbauer
derived the assertion in }(\ref{enu:stadlbauer-amenablesymmetric-fullpressure}).

\selectlanguage{english}%
We make use of this approach to give a similar result for compactly
asymptotically $\alpha$-symmetric potentials of medium variation.
Note that $\mathcal{P}\left(\varphi,\sigma\right)$ is allowed to
be infinite in the following corollary.
\selectlanguage{american}%
\begin{cor}
\label{cor:amenability-symmetry-approximation}Let $\Sigma$ be topologically
mixing Markov shift and let $\left(\Sigma\times G,\sigma\rtimes\Psi\right)$
be an irreducible group extended Markov system. Let $\varphi:\Sigma\rightarrow\R$
be of medium variation and suppose that $\varphi$ is compactly asymptotically
$\alpha$-symmetric, for some $\alpha\ge1$. If $G$ is amenable,
then $\mathcal{P}\left(\varphi\circ\pi_{1},\sigma\rtimes\Psi\right)\ge\mathcal{P}\left(\varphi,\sigma\right)-\log\alpha$. 
\end{cor}
The main ingredients in the proof of Theorem \ref{thm:amenable-spectralradius-equivalence}
are a recent result of Stadlbauer (\cite{Stadlbauer11}) and a result
of Day (\cite[Theorem 1]{MR0159230}). \foreignlanguage{english}{Proposition
\ref{prop:symmetry-pressureisspectralradius} extends \cite[Corollary  3.17]{Jaerisch11a}
and goes back to work of Pier (\cite[pp. 196-202]{MR767264}), which
was used by Gerl (\cite[p. 177]{MR938257}).}
\selectlanguage{english}%
\begin{rem}
\label{tams-comments}In \cite{Jaerisch11a}, the special case of
a finite state Markov shift $\Sigma$ and a locally constant potential
$\varphi:\Sigma\rightarrow\R$ was considered. We have investigated
the action of the Perron-Frobenius operator $\mathcal{L}_{\varphi\circ\pi_{1}}$
on the Banach space $V_{k}\subset L^{2}\left(\Sigma\times G,\mu_{\varphi}\times\lambda\right)$,
for some $k\in\N$, where $V_{k}$ consists of functions $f:\Sigma\times G\rightarrow\R$
which are constant on the cylindrical sets $\left[\omega\right]\times\left\{ g\right\} $,
$\omega\in\Sigma^{k}$, $g\in G$, and $\lambda$ denotes the Haar
measure on $G$. Note that $L^{2}\left(\Sigma\times G,\mu_{\varphi}\times\lambda\right)$
coincides with $\left(\mathcal{H}_{2},\vert\cdot\vert_{2}\right)$
defined in Section 3. Moreover, since $\min\left\{ \sqrt{\mu_{\varphi}\left(\left[\omega\right]\right)}:\omega\in\Sigma^{k}\right\} \left|f\right|_{\infty}\le\left|f\right|_{2}\le\left|f\right|_{\infty}$,
for each $f\in V_{k}$, we have that $\left(V_{k},\vert\cdot\vert_{2}\right)$
is canonically embedded in $\left(\mathcal{H}_{\infty},\vert\cdot\vert_{\infty}\right)$.
Furthermore, it follows from (\ref{eq:norm-on-Hc}) in Lemma \ref{lem:growth-is-spectralradius}
that the norm of $\mathcal{L}_{\varphi\circ\pi_{1}}^{n}:\mathcal{H}_{\infty}\rightarrow\mathcal{H}_{\infty}$
is attained on $V_{k}$, for each $n\in\N$. Hence, the spectral radii
of $\mathcal{L}_{\varphi\circ\pi_{1}}$ and $\mathcal{L}_{\varphi\circ\pi_{1}}\big|_{V_{k}}$
coincide and  the results given in \cite[Theorem 3.21]{Jaerisch11a}
follow from Proposition \ref{prop:symmetry-pressureisspectralradius}
and Theorem \ref{thm:amenable-spectralradius-equivalence}. Finally,
we remark that, for a \foreignlanguage{american}{Hölder continuous
potential $\varphi$, }the logarithm of the spectral radius of $\mathcal{L}_{\varphi\circ\pi_{1}}$
acting on $L^{2}\left(\Sigma\times G,\mu_{\varphi}\times\lambda\right)$
is always equal to $\mathcal{P}\left(\varphi,\sigma\right)$ (\cite[Theorem 3.21]{Jaerisch11a}),\foreignlanguage{american}{
in contrast to $\mathcal{L}_{\varphi\circ\pi_{1}}:\mathcal{H}_{\infty}\rightarrow\mathcal{H}_{\infty}$
considered in this paper}. \end{rem}
\selectlanguage{american}%
\begin{acknowledgement*}
The author thanks Manuel Stadlbauer and Sara Munday for valuable comments
on an ealier version of this paper. The author thanks the referee
for his/her helpful comments on the paper.
\end{acknowledgement*}

\section{Preliminaries}

Throughout, the state space for the symbolic thermodynamic formalism
will be a \emph{Markov shift }
\[
\Sigma:=\left\{ \omega:=\left(\omega_{1},\omega_{2},\ldots\right)\in I^{\N}:\; a\left(\omega_{i},\omega_{i+1}\right)=1\mbox{ for all }i\in\N\right\} 
\]
with finite or countable infinite \emph{alphabet} $I\subset\N$, \emph{incidence
matrix} $A=\left(a\left(i,j\right)\right)\in\left\{ 0,1\right\} ^{I\times I}$
and \emph{left shift map} $\sigma:\Sigma\rightarrow\Sigma$, given
by $\sigma\left(\left(\omega_{1},\omega_{2},\ldots\right)\right):=\left(\omega_{2},\omega_{3},\ldots\right)$.
The set of \emph{$A$-admissible words} of length $n\in\mathbb{N}$
is denoted by $\Sigma^{n}:=\left\{ \omega\in I^{n}:\,\, a\left(\omega_{i},\omega_{i+1}\right)=1,\mbox{ for }1\le i\le n-1\right\} $,
and the set of $A$-admissible words of arbitrary length by $\Sigma^{*}:=\bigcup_{n\in\N}\Sigma^{n}$.
Let us also define the \emph{word length function} $\left|\cdot\right|:\,\Sigma^{*}\cup\Sigma\rightarrow\N\cup\left\{ \infty\right\} $,
where for $\omega\in\Sigma^{*}$ we set $\left|\omega\right|$ to
be the unique $n\in\N$ such that $\omega\in\Sigma^{n}$, and for
$\omega\in\Sigma$ we set $\left|\omega\right|:=\infty$. For each
$\omega\in\Sigma^{*}\cup\Sigma$ and $n\in\N$ with $n\le\left|\omega\right|$,
we define $\omega_{|n}:=\left(\omega_{1},\dots,\omega_{n}\right)$.
For $\omega,\tau\in\Sigma$, we set $\omega\wedge\tau$ to be the
longest common initial block of $\omega$ and $\tau$, that is, $\omega\wedge\tau:=\omega_{|l}$,
where $l:=\sup\left\{ n\in\N:\omega_{|n}=\tau_{|n}\right\} $. For
$n\in\N$ and $\omega\in\Sigma^{n}$, we let $\left[\omega\right]:=\left\{ \tau\in\Sigma:\tau_{|n}=\omega\right\} $
denote the \emph{cylindrical set} given by $\omega$.

If $\Sigma$ is the Markov shift with alphabet $I$ whose incidence
matrix consists entirely of $1$s, then we have that $\Sigma=I^{\N}$
and $\Sigma^{n}=I^{n}$ for all $n\in\N$. Then we set $I^{*}:=\Sigma^{*}$.
For $\omega,\tau\in I^{*}$ we denote by $\omega\tau\in I^{*}$ the
\emph{concatenation} of $\omega$ and $\tau$, which is defined by
$\omega\tau:=\left(\omega_{1},\dots,\omega_{\left|\omega\right|},\tau_{1},\dots,\tau_{\left|\tau\right|}\right)$
for $\omega,\tau\in I^{*}$. Note that $I^{*}$ forms a semigroup
with respect to the concatenation operation. The semigroup $I^{*}$
is the free semigroup over the set $I$ and satisfies the following
universal property: For each semigroup $S$ and for every map $u:I\rightarrow S$,
there exists a unique semigroup homomorphism $\hat{u}:I^{*}\rightarrow S$
such that $\hat{u}\left(i\right)=u\left(i\right)$, for all $i\in I$. 

We equip $I^{\N}$ with the product topology of the discrete topology
on $I$. The Markov shift $\Sigma\subset I^{\N}$ is equipped with
the subspace topology. A countable basis of this topology on $\Sigma$
is given by the cylindrical sets $\left\{ \left[\omega\right]:\omega\in\Sigma^{*}\right\} $.
We will make use of the following metric generating the topology on
$\Sigma$. For $\beta>0$, we define the metric $d_{\beta}$ on $\Sigma$
given by 
\[
d_{\beta}\left(\omega,\tau\right):=\e^{-\beta\left|\omega\wedge\tau\right|},\mbox{ for all }\omega,\tau\in\Sigma.
\]

A function $\varphi:\Sigma\rightarrow\R$ is also called a \emph{potential}.
For $n\in\N$, we use $S_{n}\varphi:\Sigma\rightarrow\R$ to denote
the \emph{ergodic sum} of $\varphi$ with respect to $\sigma$, in
other words, $S_{n}\varphi:=\sum_{i=0}^{n-1}\varphi\circ\sigma^{i}$.
Also, we set $S_{0}\varphi:=0$. The function $\varphi$ is called
\emph{$\beta$-Hölder continuous}, for some $\beta>0$, if 
\[
V_{\beta}\left(\varphi\right):=\sup_{n\ge1}\left\{ V_{\beta,n}\left(\varphi\right)\right\} <\infty,\quad\mbox{where }V_{\beta,n}\left(\varphi\right):=\sup\left\{ \frac{\left|\varphi\left(\omega\right)-\varphi\left(\tau\right)\right|}{d_{\beta}\left(\omega,\tau\right)}:\omega,\tau\in\Sigma,\left|\omega\wedge\tau\right|\ge n\right\} ,n\in\N.
\]
We say that \emph{$\varphi$ }is\emph{ Hölder continuous} if there
exists $\beta>0$ such that $\varphi$ is $\beta$-Hölder continuous.
The function $\varphi$ is of \emph{medium variation}, if $\varphi$
is continuous and if there exists a sequence $\left(D_{n}\right)\in\R^{\N}$
with $\lim_{n}\left(D_{n}\right)^{1/n}=1$ such that, for all $n\in\N$,
$\omega\in\Sigma^{n}$ and $x,y\in\left[\omega\right]$, we have $\e^{S_{n}\varphi\left(x\right)-S_{n}\varphi\left(y\right)}\le D_{n}$. 

We need the following topological mixing properties for Markov shifts. 

\begin{defn} \label{mixing-definitions}Let $\Sigma$ be a Markov
shift with alphabet $I\subset\N$. 
\begin{itemize}
\item $\Sigma$ is \emph{irreducible} if, for all $i,j\in I$, there exists
$\omega\in\Sigma^{*}$ such that $i\omega j\in\Sigma^{*}$.
\item $\Sigma$ is \emph{topologically mixing} if, for all $i,j\in I$,
there exists $n_{0}\in\N$ with the property that, for all $n\ge n_{0}$,
there exists $\omega\in\Sigma^{n}$ such that $i\omega j\in\Sigma^{*}$.
\item $\Sigma$ satisfies the \emph{big images and preimages (b.i.p.)} \emph{property}
(\cite{MR1955261}) if there exists a finite set $\Lambda\subset I$
such that, for all $i\in I$, there exist $a,b\in\Lambda$ such that
$aib\in\Sigma^{3}$. 
\end{itemize}
\end{defn} 

The Gurevi\v{c} pressure of a Hölder continuous potential on a topologically
mixing Markov shift was introduced by Sarig (\cite[Definition 1]{MR1738951}).
One can easily verify that the Gurevi\v{c} pressure is well-defined
also for a potential of medium variation on an irreducible Markov
shift (cf. \cite[p.453]{Stadlbauer11}). The Gurevi\v{c} pressure
extends the notion of the Gurevi\v{c} entropy (\cite{MR0263162,MR0268356}),
which corresponds to the constant zero potential. 

\begin{defn} \label{def:gurevich} Let $\Sigma$ be an irreducible
Markov shift with alphabet $I$ and left shift $\sigma:\Sigma\rightarrow\Sigma$.
Let $\varphi:\Sigma\rightarrow\R$ be of medium variation. For each
$a\in I$ and $n\in\N$, we set 
\[
Z_{n}\left(\varphi,a,\sigma\right):=\sum_{\omega\in\Sigma^{n},\,\,\omega_{1}=a,\,\,\omega a\in\Sigma^{*}}\e^{\sup S_{n}\varphi_{|\left[\omega\right]}}\quad\text{and}\quad Z_{n}^{*}\left(\varphi,a,\sigma\right):=\sum_{\omega\in\Sigma^{n},\,\,\omega_{1}=a,\,\,\omega a\in\Sigma^{*},\,\,\omega_{2}\neq a,\dots,\omega_{n}\neq a}\e^{\sup S_{n}\varphi_{|\left[\omega\right]}}.
\]
The \emph{Gurevi\v{c} pressure} of $\varphi$ with respect to $\sigma$
is, for each $a\in I$, given by 
\[
\mathcal{P}\left(\varphi,\sigma\right):=\limsup_{n\rightarrow\infty}\frac{1}{n}\log Z_{n}\left(\varphi,a,\sigma\right).
\]

\end{defn} 
\begin{rem}
\label{pressure-groupextension}A group extended Markov system $\left(\Sigma\times G,\sigma\rtimes\Psi\right)$
is a Markov shift with alphabet $I\times G$. If $\left(\Sigma\times G,\sigma\rtimes\Psi\right)$
is irreducible, $\varphi:\Sigma\rightarrow\R$ is of medium variation
and $a\in I$, then 
\[
\mathcal{P}\left(\varphi\circ\pi_{1},\sigma\rtimes\Psi\right)=\limsup_{n\rightarrow\infty}\frac{1}{n}\log\sum_{\omega\in\Sigma^{n},\,\,\omega_{1}=a,\,\,\omega a\in\Sigma^{*},\,\,\Psi\left(\omega\right)=\id}\e^{\sup S_{n}\varphi_{|\left[\omega\right]}}.
\]

\end{rem}
The next definition goes back to the work of Ruelle and Bowen (cf.
\cite{MR0289084}, \cite{bowenequilibriumMR0442989}).
\begin{defn}
\label{gibbs-measure}Let $\Sigma$ be a Markov shift and let $\varphi:\Sigma\rightarrow\R$
be Hölder continuous with $\mathcal{P}\left(\varphi,\sigma\right)<\infty$.
We say that a Borel probability measure \emph{$\mu$ }is a\emph{ Gibbs
measure for $\varphi$ }if there exists a constant $C_{\mu}\ge1$
such that 
\begin{equation}
C_{\mu}^{-1}\le\frac{\mu\left(\left[\omega\right]\right)}{\e^{S_{\left|\omega\right|}\varphi\left(\tau\right)-\left|\omega\right|\mathcal{P}\left(\varphi,\sigma\right)}}\le C_{\mu},\mbox{ for all }\omega\in\Sigma^{*}\mbox{ and }\tau\in\left[\omega\right].\label{eq:gibbs-equation}
\end{equation}

\end{defn}
The following criterion for the existence of Gibbs measures is taken
from \cite[Theorem 1]{MR1955261}. The uniqueness follows from \cite[Theorem 2.2.4]{MR2003772}.
The existence of a $\sigma$-invariant Gibbs measure and a Hölder
continuous eigenfunction on a topologically mixing Markov shift with
the b.i.p. property follows from \cite{MR2003772}.
\begin{thm}
\label{thm:existence-of-gibbs-measures}Let $\Sigma$ be a topologically
mixing Markov shift and let $\varphi:\Sigma\rightarrow\R$ be Hölder
continuous. Then there exists a (unique) $\sigma$-invariant Gibbs
measure $\mu_{\varphi}$ for $\varphi$ if and only if $\Sigma$ satisfies
the b.i.p. property and\textup{ $\mathcal{P}\left(\varphi,\sigma\right)<\infty$.}
Moreover, if $\Sigma$ satisfies the b.i.p. property and \textup{$\mathcal{P}\left(\varphi,\sigma\right)<\infty$,}
then there exists a unique Hölder continuous function $h:\Sigma\rightarrow\R^{+}$,
bounded away from zero and infinity, such that $\mathcal{L}_{\varphi}\left(h\right)=\e^{\mathcal{P}\left(\varphi,\sigma\right)}h$
and $\int\psi\, d\mu_{\varphi}=\int\mathcal{L}_{\varphi+\log h-\log h\circ\sigma-\mathcal{P}\left(\varphi,\sigma\right)}\left(\psi\right)\, d\mu_{\varphi}$,
for every bounded continuous function $\psi:\Sigma\rightarrow\R$. 
\end{thm}

\section{Proofs}

Let us first state the necessary definitions and notations which are
needed for the proofs of our main results. If $\Sigma$ is a topologically
mixing Markov shift with the b.i.p. property, and if $\varphi:\Sigma\rightarrow\R$
is Hölder continuous with $\mathcal{P}\left(\varphi,\sigma\right)<\infty$,
then there exists a unique $\sigma$-invariant Gibbs measure $\mu_{\varphi}$
for $\varphi$ by Theorem \ref{thm:existence-of-gibbs-measures}.
For $p\in\N\cup\left\{ \infty\right\} $ and $\psi\in L^{p}\left(\Sigma,\mathcal{B}\left(\Sigma\right),\mu_{\varphi}\right)$,
we denote by $\Vert\psi\Vert_{p}$ the $L^{p}$-norm of $\psi$, where
$\mathcal{B}\left(\Sigma\right)$ is the Borel sigma algebra of $\Sigma$.
For a group extended Markov system $\left(\Sigma\times G,\sigma\rtimes\Psi\right)$,
Stadlbauer (\cite{Stadlbauer11}) introduced the Banach space $\big(\mathcal{H}_{p},\left|\cdot\right|_{p}\big)$,
given by 
\[
\mathcal{H}_{p}:=\left\{ f:\Sigma\times G\rightarrow\R:\left|f\right|_{p}<\infty\right\} ,\quad\mbox{where }\left|f\right|_{p}:=\sqrt{\sum_{g\in G}\left(\Vert f\left(\cdot,g\right)\Vert_{p}\right)^{2}}.
\]
Denote by $\1_{\Omega}$ the indicator function of a set $\Omega\subset\Sigma\times G$.
Let $\mathcal{H}_{c}$ be the closed subspace of $\mathcal{H}_{\infty}$
generated by $\left\{ \1_{\Sigma\times\left\{ g\right\} }:g\in G\right\} $,
which is isomorphic to the Hilbert space $\ell^{2}\left(G\right)$.
We use \foreignlanguage{english}{$\left(\cdot,\cdot\right)$ to denote
the inner product on }$\mathcal{H}_{c}$, given by\foreignlanguage{english}{
$\left(f_{1},f_{2}\right):=\sum_{g\in G}\int f_{1}(\cdot,g)f_{2}(\cdot,g)d\mu_{\varphi}$,
for all $f_{1},f_{2}\in\mathcal{H}_{c}$.} For a closed subspace $\mathcal{V}\subset\mathcal{H}_{\infty}$,
we set $\mathcal{V}^{+}:=\left\{ f\in\mathcal{V}:f\ge0\right\} $.
For closed subspaces $\mathcal{V}_{1},\mathcal{V}_{2}\subset\mathcal{H}_{\infty}$
and a bounded linear operator $T:\mathcal{V}_{1}\rightarrow\mathcal{V}_{2}$,
the operator norm of $T$ is given by $\Vert T\Vert:=\sup_{f\in\mathcal{V}_{1},\vert f\vert_{\infty}=1}\vert T\left(f\right)\vert_{\infty}$.
We say that $T$ is positive if $T\left(\mathcal{V}_{1}^{+}\right)\subset\mathcal{V}_{2}^{+}$.
It is well-known that, if $T$ is positive and $\mathcal{V}=\mathcal{V}^{+}-\mathcal{V}^{+}$,
then $\Vert T\Vert=\sup_{f\in\mathcal{V}_{1}^{+},\vert f\vert_{\infty}=1}\vert T\left(f\right)\vert_{\infty}$. 

It is shown in \cite[Proposition 5.2]{Stadlbauer11} that the Perron-Frobenius
operator $\mathcal{L}_{\varphi\circ\pi_{1}}:\mathcal{H}_{\infty}\rightarrow\mathcal{H}_{\infty}$
is a bounded linear operator. Further, if $\mathcal{L}_{\varphi}\left(\1\right)=\1$,
then 
\begin{equation}
\Vert\mathcal{L}_{\varphi\circ\pi_{1}}^{n}\Vert\le C,\quad n\in\N,\label{eq:potenzstetig}
\end{equation}
where $C=C_{\mu_{\varphi}}\ge1$ denotes the constant of the Gibbs
measure $\mu_{\varphi}$ given by (\ref{eq:gibbs-equation}). 
\begin{defn}
We define the linear operators \foreignlanguage{english}{$\A:\mathcal{H}_{\infty}\rightarrow\mathcal{H}_{c}$
and }$T_{n}:\mathcal{H}_{c}\rightarrow\mathcal{H}_{c}$\foreignlanguage{english}{,
which are for $f\in\mathcal{H}_{\infty}$ and $n\in\N$ given by}
\[
\A\left(f\right):=\sum_{g\in G}\left(\int f\left(\cdot,g\right)d\mu_{\varphi}\right)\1_{\Sigma\times\left\{ g\right\} }\quad\mbox{ and }\quad T_{n}:=A\mathcal{L}_{\varphi\circ\pi_{1}}^{n}\big|_{\mathcal{H}_{c}.}
\]
Clearly, the linear operators $A$ and $T_{n}$ are positive and bounded
with $\Vert A\Vert=1$ and $\Vert T_{n}\Vert\le\Vert A\Vert\Vert\mathcal{L}_{\varphi\circ\pi_{1}}^{n}\Vert=\Vert\mathcal{L}_{\varphi\circ\pi_{1}}^{n}\Vert$,
for each $n\in\N$. To state the next lemma, let us recall the definition
of \foreignlanguage{english}{$\Lambda_{n}$ (}\cite[Proposition 5.2]{Stadlbauer11}\foreignlanguage{english}{),
which is given by} $\Lambda_{n}:=\sup_{f\in\mathcal{H}_{c}^{+},\left|f\right|_{\infty}=1}\left|\mathcal{L}_{\varphi\circ\pi_{1}}^{n}\left(f\right)\right|_{1}$,
\foreignlanguage{english}{for each $n\in\N$}. \end{defn}
\begin{lem}
\label{lem:growth-is-spectralradius}Under the assumptions of Theorem
\ref{thm:amenable-spectralradius-equivalence}, suppose that $\mathcal{L}_{\varphi}\left(\1\right)=\1$.
Then we have the following.
\begin{enumerate}
\item \label{enu:Lambdan-is-Tnnorm}$\Vert T_{n}\Vert=\Lambda_{n}$, $n\in\N$. 
\selectlanguage{english}%
\item \textup{\label{enu:Ln-Tn-norm}$\Vert T_{n}\Vert\le\Vert\mathcal{L}_{\varphi\circ\pi_{1}}^{n}\Vert\le C\Vert T_{n}\Vert$,
$n\in\N$. }
\selectlanguage{american}%
\item \label{enu:spectralradius-lambdan}$\lim_{n\rightarrow\infty}\Vert T_{n}\Vert^{1/n}=\lim_{n\rightarrow\infty}\left(\Lambda_{n}\right)^{1/n}=\rho\left(\mathcal{L}_{\varphi\circ\pi_{1}}\right)$. 
\selectlanguage{english}%
\item \textup{\label{enu:gurevich-vs-Tn}$\limsup_{n\rightarrow\infty}n^{-1}\log\left(T_{n}\1_{\Sigma\times\left\{ \id\right\} },\1_{\Sigma\times\left\{ \id\right\} }\right)=\mathcal{P}\left(\varphi\circ\pi_{1},\sigma\rtimes\Psi\right)$. }
\end{enumerate}
\end{lem}
\begin{proof}
Let $n\in\N$. The assertion in (\ref{enu:Lambdan-is-Tnnorm}) follows
because $T_{n}$ is positive and $\left|T_{n}\left(f\right)\right|_{\infty}=\left|\mathcal{L}_{\varphi\circ\pi_{1}}^{n}\left(f\right)\right|_{1}$,
for each $f\in\mathcal{H}_{c}^{+}$. The first inequality in (\ref{enu:Ln-Tn-norm})
holds since $\Vert A\Vert=1$. To prove the second inequality we show
that 
\begin{equation}
\Vert\mathcal{L}_{\varphi\circ\pi_{1}}^{n}\Vert=\sup\left\{ \left|\mathcal{L}_{\varphi\circ\pi_{1}}^{n}\left(f\right)\right|_{\infty}:f\in\mathcal{H}_{c}^{+},\left|f\right|_{\infty}=1\right\} .\label{eq:norm-on-Hc}
\end{equation}
For each $f\in\mathcal{H}_{\infty}^{+}$, there exists $\tilde{f}\in\mathcal{H}_{c}^{+}$,
given by $\tilde{f}\left(x,g\right):=\Vert f\left(\cdot,g\right)\Vert_{\infty}$,
$\left(x,g\right)\in\Sigma\times G$, which satisfies $\left|\tilde{f}\right|_{\infty}=\left|f\right|_{\infty}$
and $\left|\mathcal{L}_{\varphi\circ\pi_{1}}^{n}\left(f\right)\right|_{\infty}\le\left|\mathcal{L}_{\varphi\circ\pi_{1}}^{n}\left(\tilde{f}\right)\right|_{\infty}$.
Hence, (\ref{eq:norm-on-Hc}) follows. To prove $\Vert\mathcal{L}_{\varphi\circ\pi_{1}}^{n}\Vert\le C\Vert T_{n}\Vert$,
let $f\in\mathcal{H}_{c}^{+}$. \foreignlanguage{english}{By the Gibbs
property }(\ref{eq:gibbs-equation})\foreignlanguage{english}{ of
$\mu_{\varphi}$, there exists $C\ge1$ such that, }for each $g\in G$
and $x_{0}\in\Sigma$, 
\[
\Vert\mathcal{L}_{\varphi\circ\pi_{1}}^{n}\left(f\right)\left(\cdot,g\right)\Vert_{\infty}\le\sum_{\omega\in\Sigma^{n}}\sup_{\tau\in\Sigma:\omega\tau\in\Sigma}\e^{S_{n}\varphi\left(\omega\tau\right)}f\big(x_{0},g\Psi(\omega)^{-1}\big)\le C\sum_{\omega\in\Sigma^{n}}\mu_{\varphi}\left(\left[\omega\right]\right)f\big(x_{0},g\Psi\left(\omega\right)^{-1}\big).
\]
By \foreignlanguage{english}{Theorem \ref{thm:existence-of-gibbs-measures}}
we have $\mu_{\varphi}\left(\left[\omega\right]\right)=\int\mathcal{L}_{\varphi}^{n}\left(\1_{\left[\omega\right]}\right)\, d\mu_{\varphi}$,
for each $\omega\in\Sigma^{n}$. Hence, we obtain that 
\begin{align}
 & \quad\,\,\,\sum_{\omega\in\Sigma^{n}}\mu_{\varphi}\left(\left[\omega\right]\right)f\big(x_{0},g\Psi\left(\omega\right)^{-1}\big)=\sum_{\omega\in\Sigma^{n}}\int\mathcal{L}_{\varphi}^{n}\left(\1_{\left[\omega\right]}\right)f\big(x_{0},g\Psi\left(\omega\right)^{-1}\big)\, d\mu_{\varphi}\label{eq:1-norm-vs-gibbs}\\
 & =\sum_{\omega\in\Sigma^{n}}\int\mathcal{L}_{\varphi\circ\pi_{1}}^{n}\big(\1_{\left[\omega\right]\times G}f\big)\left(\cdot,g\right)\, d\mu_{\varphi}=\int\mathcal{L}_{\varphi\circ\pi_{1}}^{n}\left(f\right)\left(\cdot,g\right)\, d\mu_{\varphi}=\Vert\mathcal{L}_{\varphi\circ\pi_{1}}^{n}\left(f\right)\left(\cdot,g\right)\Vert_{1},\nonumber 
\end{align}
which proves \foreignlanguage{english}{$\left|\mathcal{L}_{\varphi\circ\pi_{1}}^{n}\left(f\right)\right|_{\infty}\le C\left|\mathcal{L}_{\varphi\circ\pi_{1}}^{n}\left(f\right)\right|_{1}=C\left|T_{n}\left(f\right)\right|_{\infty}$}
and finishes  the proof of (\ref{enu:Ln-Tn-norm}). 

The assertions in (\ref{enu:spectralradius-lambdan}) follow by combining
(\ref{enu:Lambdan-is-Tnnorm}), (\ref{enu:Ln-Tn-norm}) and Gelfand's
formula for the spectral radius (see e.g. \cite[Theorem 10.13]{MR0365062}).
Let us now turn to the proof of (\ref{enu:gurevich-vs-Tn}). By \foreignlanguage{english}{(\ref{eq:1-norm-vs-gibbs}),
we see that }$T_{n}$ is for each \foreignlanguage{english}{$n\in\N$,
$f\in\mathcal{H}_{c}$ and $x_{0}\in\Sigma$ given by
\begin{equation}
T_{n}\left(f\right)=\sum_{g\in G}\left(\sum_{\omega\in\Sigma^{n}}\mu_{\varphi}\left(\left[\omega\right]\right)f\big(x_{0},g\Psi\left(\omega\right)^{-1}\big)\right)\1_{\Sigma\times\left\{ g\right\} }.\label{eq:Tn-is-convolution}
\end{equation}
 }Consequently, by the Gibbs property (\ref{eq:gibbs-equation}) of
$\mu_{\varphi}$ we have 
\[
C^{-1}\sum_{\omega\in\Sigma^{n}:\Psi\left(\omega\right)=\id}\e^{\sup S_{n}\varphi_{|\left[\omega\right]}}\le\left(T_{n}\1_{\Sigma\times\left\{ \id\right\} },\1_{\Sigma\times\left\{ \id\right\} }\right)\le C\sum_{\omega\in\Sigma^{n}:\Psi\left(\omega\right)=\id}\e^{\inf S_{n}\varphi_{|\left[\omega\right]}}.
\]
It follows that $\mathcal{P}\left(\varphi\circ\pi_{1},\sigma\rtimes\Psi\right)\le\limsup_{n\rightarrow\infty}n^{-1}\log\left(T_{n}\1_{\Sigma\times\left\{ \id\right\} },\1_{\Sigma\times\left\{ \id\right\} }\right)$
(cf. Remark \ref{pressure-groupextension}). To prove equality, we
fix $a\in I$. Since $\left(\Sigma\times G,\sigma\rtimes\Psi\right)$
is irreducible, there exists $\kappa\left(a\right)\in\Sigma^{*}$
such that $\Psi\left(a\right)\Psi\left(\kappa\left(a\right)\right)=\id$.
\foreignlanguage{english}{Since $\Sigma$ has the b.i.p. property
and }$\left(\Sigma\times G,\sigma\rtimes\Psi\right)$ is irreducible\foreignlanguage{english}{,
there exists a finite set $F\subset\Psi^{-1}\left(\id\right)\cap\Sigma^{*}$
such that, for all $i,j\in I$ there exists $\gamma\in F$ with $i\gamma j\in\Sigma^{*}$.
Set $l:=\max_{\gamma\in F}\left|\gamma\right|$. For each $n\in\N$
and $\omega\in\Sigma^{n}$, there exist $\gamma_{1},\gamma_{2},\gamma_{3}\in F$
such that $\tau\left(\omega\right):=a\gamma_{1}\kappa\left(a\right)\gamma_{2}\omega\gamma_{3}a\in\Sigma^{*}$.
This defines a map from $\left\{ \omega\in\Sigma^{n}:\Psi\left(\omega\right)=\id\right\} $
to $\left\{ \tau\in\Sigma^{k}:n\le k\le n+3l+\left|\kappa\left(a\right)\right|+1,\Psi\left(\tau\right)=\id,\tau_{1}=a,\tau a\in\Sigma^{*}\right\} $,
which is at most $\left(3l+\left|\kappa\left(a\right)\right|\right)$-to-one,
for each $n\in\N$. Setting $M:=\e^{-\inf\varphi_{|\left[a\right]}}\e^{-\inf S_{\left|\kappa\left(a\right)\right|}\varphi_{|\left[\kappa\left(a\right)\right]}}\max_{\gamma\in F}\e^{-3\inf S_{\left|\gamma\right|}\varphi_{|\left[\gamma\right]}}$,
we obtain that, for each $n\in\N$, 
\[
\sum_{\omega\in\Sigma^{n}:\Psi\left(\omega\right)=\id}\e^{\inf S_{n}\varphi_{|\left[\omega\right]}}\le M\left(3l+\left|\kappa\left(a\right)\right|\right)\sum_{k=n}^{n+3l+\left|\kappa\left(a\right)\right|+1}\sum_{\tau\in\Sigma^{k}:\Psi\left(\tau\right)=\id,\tau_{1}=a,\tau a\in\Sigma^{*}}\e^{\sup S_{k}\varphi_{|\left[\tau\right]}},
\]
which gives $\limsup_{n\rightarrow\infty}n^{-1}\log\left(T_{n}\1_{\Sigma\times\left\{ \id\right\} },\1_{\Sigma\times\left\{ \id\right\} }\right)\le\mathcal{P}\left(\varphi\circ\pi_{1},\sigma\rtimes\Psi\right)$
and finishes the proof of (\ref{enu:gurevich-vs-Tn}).}
\end{proof}

\selectlanguage{english}%
\begin{proof}
[Proof of Theorem $\ref{thm:amenable-spectralradius-equivalence}$]Let
us first verify that we may assume without loss of generality that
$\mathcal{L}_{\varphi}\left(\1\right)=\1$ and hence, $\mathcal{P}\left(\varphi,\sigma\right)=0$.
Otherwise, there exists a H\"older continuous function $h:\Sigma\rightarrow\R^{+}$,
bounded away from zero and infinity, such that $\mathcal{L}_{\varphi}\left(h\right)=\e^{\mathcal{P}\left(\varphi,\sigma\right)}h$
by Theorem \ref{thm:existence-of-gibbs-measures}. Setting $\tilde{\varphi}:=\varphi+\log h-\log h\circ\sigma-\mathcal{P}\left(\varphi,\sigma\right)$,
we have $\mathcal{L}_{\tilde{\varphi}}\left(\1\right)=\1$ and $\mathcal{P}\left(\tilde{\varphi},\sigma\right)=0$.
Since $\mathcal{L}_{\tilde{\varphi}\circ\pi_{1}}\left(f\right)=\e^{-\mathcal{P}\left(\varphi,\sigma\right)}\frac{1}{h\circ\pi_{1}}\mathcal{L}_{\varphi\circ\pi_{1}}\left(f\left(h\circ\pi_{1}\right)\right)$,
for each $f\in\mathcal{H}_{\infty}$, and using that $h\circ\pi_{1}$
is bounded away from zero and infinity, we obtain that $\mathcal{L}_{\tilde{\varphi}\circ\pi_{1}}$
and $\e^{-\mathcal{P}\left(\varphi,\sigma\right)}\mathcal{L}_{\varphi\circ\pi_{1}}$
have the same spectrum. Hence, we have $\log\rho\left(\mathcal{L}_{\tilde{\varphi}\circ\pi_{1}}\right)=\log\rho\left(\mathcal{L}_{\varphi\circ\pi_{1}}\right)-\mathcal{P}\left(\varphi,\sigma\right)$.
We have thus shown that we may assume without loss of generality that
$\mathcal{L}_{\varphi}\left(\1\right)=\1$. 

\selectlanguage{american}%
That $\rho\left(\mathcal{L}_{\varphi\circ\pi_{1}}\right)\le1$ follows
from (\ref{eq:potenzstetig}) and Gelfand's formula for the spectral
radius. \foreignlanguage{english}{We now turn to the proof of the
amenability dichotomy. First suppose that $\rho\left(\mathcal{L}_{\varphi\circ\pi_{1}}\right)=1$.
By Lemma \ref{lem:growth-is-spectralradius} (\ref{enu:spectralradius-lambdan})
we obtain that $\lim_{n\rightarrow\infty}\left(\Lambda_{n}\right)^{1/n}=1$.
It follows from \cite[Lemma 5.3, Theorem 5.4]{Stadlbauer11} that
$G$ is amenable. In fact, the assumption that $\e^{\mathcal{P}\left(\varphi\circ\pi_{1},\sigma\rtimes\Psi\right)}=1$
in \cite[Lemma 5.3, Theorem 5.4]{Stadlbauer11} can be relaxed to
$\lim_{n\rightarrow\infty}\left(\Lambda_{n}\right)^{1/n}=1$ without
affecting the proofs. }

\selectlanguage{english}%
To prove the converse implication, suppose that $G$ is amenable.
\foreignlanguage{american}{It follows from a well-known result of
Day (\cite[Theorem 1(d)]{MR0159230}) that $\Vert T_{n}\Vert=1$,
for each $n\in\N$. To prove this, observe that by (\ref{eq:Tn-is-convolution}),
we have that $T_{n}$ is the right convolution operator with respect
to the probability density on $G$, given by $\mu_{\varphi}\big(\left\{ x\in\Sigma:\Psi\left(x_{1},\dots,x_{n}\right)=g\right\} \big)$,
for each $g\in G$. Finally, we conclude that }$\rho\left(\mathcal{L}_{\varphi\circ\pi_{1}}\right)=1$\foreignlanguage{american}{
by Lemma \ref{lem:growth-is-spectralradius} (\ref{enu:spectralradius-lambdan})}.
The proof is complete.
\end{proof}

\begin{proof}
[Proof of Proposition $\ref{prop:symmetry-pressureisspectralradius}$]\foreignlanguage{american}{
By the arguments given at the beginning of the proof of Theorem \ref{thm:amenable-spectralradius-equivalence}
and by  Remark \ref{rem:symmetry-coboundary}, we may assume that
$\mathcal{L}_{\varphi}\left(\1\right)=\1$. To prove $\mathcal{P}\left(\varphi\circ\pi_{1},\sigma\rtimes\Psi\right)\le\log\rho\left(\mathcal{L}_{\varphi\circ\pi_{1}}\right)$,
we observe that, by Lemma \ref{lem:growth-is-spectralradius} (\ref{enu:gurevich-vs-Tn}),
the Cauchy-Schwarz inequality and Lemma \ref{lem:growth-is-spectralradius}
}(\ref{enu:spectralradius-lambdan})\foreignlanguage{american}{, we
have 
\[
\mathcal{P}\left(\varphi\circ\pi_{1},\sigma\rtimes\Psi\right)=\limsup_{n\rightarrow\infty}\frac{1}{n}\log\left(T_{n}\1_{\Sigma\times\left\{ \id\right\} },\1_{\Sigma\times\left\{ \id\right\} }\right)\le\limsup_{n\rightarrow\infty}\frac{1}{n}\log\Vert T_{n}\Vert=\log\rho\left(\mathcal{L}_{\varphi\circ\pi_{1}}\right).
\]
}Now suppose that $\varphi$ is asymptotically $\alpha$-symmetric
with respect to $\Psi$. \foreignlanguage{american}{The main task
is to prove that $\lim_{n\rightarrow\infty}\Vert T_{n}\Vert^{1/n}\le\alpha\e^{\mathcal{P}\left(\varphi\circ\pi_{1},\sigma\rtimes\Psi\right)}$,
from which t}he proposition follows by Lemma \ref{lem:growth-is-spectralradius}
(\ref{enu:spectralradius-lambdan}). Since $T_{n}^{*}T_{n}$ is a
self-adjoint operator on the Hilbert space $\mathcal{H}_{c}$, an
argument of Pier (\cite[pp. 196-202]{MR767264}), which was used by
\foreignlanguage{american}{Gerl} (\cite[p. 177]{MR938257}), shows
that for each $n\in\N$, 
\begin{equation}
\Vert T_{n}^{*}T_{n}\Vert=\limsup_{k\rightarrow\infty}\left(\left(T_{n}^{*}T_{n}\right)^{k}\left(\1_{\Sigma\times\left\{ \id\right\} }\right),\1_{\Sigma\times\left\{ \id\right\} }\right)^{1/k}.\label{eq:gerltrick-1}
\end{equation}
By (\ref{eq:Tn-is-convolution}) and the Gibbs property (\ref{eq:gibbs-equation})
of $\mu_{\varphi}$, we have for all $g_{1},g_{2}\in G$ and $n\in\N$
that 
\begin{equation}
C^{-1}\sum_{\omega\in\Sigma^{n}:g_{1}\Psi\left(\omega\right)=g_{2}}\e^{\sup S_{n}\varphi_{|\left[\omega\right]}}\le\left(T_{n}\1_{\Sigma\times\left\{ g_{1}\right\} },\1_{\Sigma\times\left\{ g_{2}\right\} }\right)\le C\sum_{\omega\in\Sigma^{n}:g_{1}\Psi\left(\omega\right)=g_{2}}\e^{\inf S_{n}\varphi_{|\left[\omega\right]}}.\label{eq:gerltrick-2}
\end{equation}
Since $\varphi$ is asymptotically $\alpha$-symmetric with respect
to $\Psi$, there exist $n_{0}\in\N$ and sequences $\left(c_{n}\right)\in\left(\R^{+}\right)^{\N}$
and $\left(N_{n}\right)\in\N^{\N}$ with the property that $\lim_{n}(c_{n})^{1/(2n)}=\alpha$
and $\lim_{n}n^{-1}N_{n}=0$, such that, for all $f_{1},f_{2}\in\mathcal{H}_{c}^{+}$
and $n>\max\left\{ N_{n},n_{0}\right\} $, we have
\begin{equation}
\left(T_{n}\left(f_{1}\right),f_{2}\right)\le c_{n}C^{2}\sum_{i=-N_{n}}^{N_{n}}\left(T_{n+i}^{*}\left(f_{1}\right),f_{2}\right).\label{eq:gerltrick-3}
\end{equation}
Since $\Sigma$ has the b.i.p. property and $\left(\Sigma\times G,\sigma\rtimes\Psi\right)$
is irreducible, there exists a finite set $F\subset\Psi^{-1}\left(\id\right)\cap\Sigma^{*}$
such that, for all $i,j\in I$ there exists $\gamma\in F$ with $i\gamma j\in\Sigma^{*}$.
Set $l:=\max_{\gamma\in F}\left|\gamma\right|$ and $M:=\max_{\gamma\in F}\e^{-\inf S_{\left|\gamma\right|}\varphi_{|\left[\gamma\right]}}$.
It follows from (\ref{eq:gerltrick-2}) that, for all $t,u\in\N$
and for all $g_{1},g_{2}\in G$, we have 
\begin{align*}
\left(T_{t}T_{u}\1_{\Sigma\times\left\{ g_{1}\right\} },\1_{\Sigma\times\left\{ g_{2}\right\} }\right) & \le C^{2}\sum_{\omega_{1}\in\Sigma^{u},\omega_{2}\in\Sigma^{t}:g_{1}\Psi\left(\omega_{1}\right)\Psi\left(\omega_{2}\right)=g_{2}}\e^{\inf S_{u}\varphi_{|\left[\omega_{1}\right]}}\e^{\inf S_{t}\varphi_{|\left[\omega_{2}\right]}}\\
 & \le MC^{2}\sum_{i=0}^{l}\sum_{\omega\in\Sigma^{u+t+i}:g_{1}\Psi\left(\omega\right)=g_{2}}\e^{\inf S_{u+t+i}\varphi_{|\left[\omega\right]}}\le MC^{3}\sum_{i=0}^{l}\left(T_{u+t+i}\1_{\Sigma\times\left\{ g_{1}\right\} },\1_{\Sigma\times\left\{ g_{2}\right\} }\right).
\end{align*}
Setting $D:=MC^{3}$, we have thus shown that, for all $f_{1},f_{2}\in\mathcal{H}_{c}^{+}$,
\begin{equation}
\left(T_{t}T_{u}\left(f_{1}\right),f_{2}\right)\le D\sum_{i=0}^{l}\left(T_{u+t+i}\left(f_{1}\right),f_{2}\right).\label{eq:gerltrick-4}
\end{equation}
We now follow \cite[Proposition 3.11]{Jaerisch11a} which was motivated
by \cite{MR2338235}. Let  $n>\max\left\{ N_{n},n_{0}\right\} $.
Combining first (\ref{eq:gerltrick-1}) and (\ref{eq:gerltrick-3}),
then by applying $(2k-1)$- times the estimate in (\ref{eq:gerltrick-4}),
we obtain that
\begin{align*}
\Vert T_{n}^{*}T_{n}\Vert & \le\limsup_{k\rightarrow\infty}\left(c_{n}^{k}C^{2k}\left(\sum_{i_{1}=-N_{n}}^{N_{n}}\left(T_{n+i_{1}}T_{n}\right)\dots\sum_{i_{k}=-N_{n}}^{N_{n}}\left(T_{n+i_{k}}T_{n}\right)\right)\left(\1_{\Sigma\times\left\{ \id\right\} }\right),\1_{\Sigma\times\left\{ \id\right\} }\right)^{1/k}\\
 & \le c_{n}C^{2}\limsup_{k\rightarrow\infty}\left(\left(2N_{n}+1\right)^{k}D^{2k-1}\left(l+1\right)^{2k-1}\max_{r=-kN_{n},\dots,k(N_{n}+2l)-l}\{(T_{2nk+r}\left(\1_{\Sigma\times\left\{ \id\right\} }\right),\1_{\Sigma\times\left\{ \id\right\} })\}\right)^{1/k}.
\end{align*}
Combining the previous estimate with $\Vert T_{n}^{*}T_{n}\Vert=\Vert T_{n}\Vert^{2}$,
$\lim_{n}(c_{n})^{1/(2n)}=\alpha$, $\lim_{n}n^{-1}N_{n}=0$ and the
fact that $\limsup_{k}\left(T_{k}\1_{\Sigma\times\left\{ \id\right\} },\1_{\Sigma\times\left\{ \id\right\} }\right)^{1/k}=\e^{\mathcal{P}\left(\varphi\circ\pi_{1},\sigma\rtimes\Psi\right)}$
by Lemma \ref{lem:growth-is-spectralradius} (\ref{enu:gurevich-vs-Tn}),
we obtain that 
\[
\lim_{n\rightarrow\infty}\Vert T_{n}\Vert^{1/n}\le\alpha\limsup_{n\rightarrow\infty}\left(\max\{\e^{\mathcal{P}\left(\varphi\circ\pi_{1},\sigma\rtimes\Psi\right)(2n-N_{n})},\e^{\mathcal{P}\left(\varphi\circ\pi_{1},\sigma\rtimes\Psi\right)(2n+N_{n}+2l)}\}\right)^{1/2n}=\alpha\e^{\mathcal{P}\left(\varphi\circ\pi_{1},\sigma\rtimes\Psi\right)}.
\]
The proof is complete.\end{proof}
\selectlanguage{american}%
\begin{lem}
\label{lem:approximation}Let $\Sigma$ be a Markov shift and let
$\varphi:\Sigma\rightarrow\R$ be of medium variation. Suppose that
$\varphi$ is asymptotically $\alpha$-symmetric with respect to $\Psi$,
for some $\alpha\ge1$. Then there exists a sequence of Hölder continuous
functions $\varphi_{j}:\Sigma\rightarrow\R$ and $\left(D_{j}\right)\in\R^{\N}$,
$D_{j}\ge1$, $j\in\N$, such that $\varphi_{j}$ is asymptotically
$\alpha(D_{j})^{1/(2j)}$-symmetric, $\lim_{j}(D_{j})^{1/(2j)}=1$
and $\lim_{j}\varphi_{j}=\varphi$, where the convergence of $\left(\varphi_{j}\right)$
is uniformly on compact subsets of $\Sigma$. \end{lem}
\begin{proof}
Define $\varphi_{j}\left(x\right):=\inf\left\{ \varphi\left(y\right):y\in\left[x_{1},\dots,x_{j}\right]\right\} $,
for each $x\in\Sigma$ and $j\in\N$. Since $\varphi$ is of medium
variation, there exists a sequence $\left(D_{j}\right)\in\R^{\N}$,
$D_{j}\ge1$, such that, for all $j,n\in\N$, $\omega\in\Sigma^{n}$
and $x\in\left[\omega\right]$, 
\begin{equation}
1\le\e^{S_{n}\varphi\left(x\right)-S_{n}\varphi_{j}\left(x\right)}\le(D_{j})^{\lfloor\frac{n}{j}\rfloor}\big(\sup_{k<j}D_{k}\big),\label{eq:mediumvar}
\end{equation}
where $\lfloor u\rfloor$ denotes the largest integer not greater
than $u$. Since $\varphi$ is asymptotically $\alpha$-symmetric,
there exist $n_{0}\in\N$ and sequences $\left(c_{n}\right)\in\R^{\N}$
and $\left(N_{n}\right)\in\N^{\N}$ with the property that $\lim_{n}\left(c_{n}\right)^{1/(2n)}=\alpha$,
$\lim_{n}n^{-1}N_{n}=0$ and such that (\ref{eq:asympt-symmetric})
holds, for each $g\in G$ and for all $n\ge n_{0}$. Let $j\in\N$.
By (\ref{eq:asympt-symmetric}) and (\ref{eq:mediumvar}) we have
for each $n\in\N$ with $n>\max\left\{ N_{n},n_{0}\right\} $ and
$g\in G$, 
\begin{align*}
\sum_{\omega\in\Sigma^{n}:\Psi\left(\omega\right)=g}\e^{\sup S_{n}\varphi_{j|\left[\omega\right]}} & \le\sum_{\omega\in\Sigma^{n}:\Psi\left(\omega\right)=g}\e^{\sup S_{n}\varphi_{|\left[\omega\right]}}\le c_{n}\sum_{\omega\in\Sigma^{*}:\Psi\left(\omega\right)=g^{-1},\, n-N_{n}\le\left|\omega\right|\le n+N_{n}}\e^{\sup S_{\left|\omega\right|}\varphi_{|\left[\omega\right]}}\\
 & \le c_{n}\max_{n-N_{n}\le l\le n+N_{n}}(D_{j})^{\lfloor\frac{l}{j}\rfloor}\big(\sup_{k<j}D_{k}\big)\sum_{\omega\in\Sigma^{*}:\Psi\left(\omega\right)=g^{-1}\, n-N_{n}\le\left|\omega\right|\le n+N_{n}}\e^{\sup S_{\left|\omega\right|}\varphi_{j|\left[\omega\right]}}.
\end{align*}
 Since $\lim_{n}\big(c_{n}\max_{n-N_{n}\le l\le n+N_{n}}(D_{j})^{\lfloor\frac{l}{j}\rfloor}\left(\sup_{k<j}D_{k}\right)\big)^{1/(2n)}=\alpha(D_{j})^{1/(2j)}$,
we have that $\varphi_{j}$ is asymptotically $\alpha(D_{j})^{1/(2j)}$-symmetric.
By continuity of $\varphi_{j}$ and $\varphi$, and using that $\varphi_{j}\le\varphi_{j+1}$,
for each $j\in\N$, we have that $\lim_{j}\varphi_{j}=\varphi$ locally
uniformly by Dini's Theorem. 
\end{proof}

\selectlanguage{english}%
\begin{proof}
[Proof of Corollary $\ref{cor:amenability-symmetry-approximation}$]

\selectlanguage{american}%
First suppose that the alphabet of $\Sigma$ is finite. Then $\mathcal{P}\left(\varphi,\sigma\right)<\infty$
and $\Sigma$ satisfies  the b.i.p. property. By Lemma \ref{lem:approximation},
there exists a sequence of Hölder continuous functions $\varphi_{j}:\Sigma\rightarrow\R$,
$j\in\N$, such that $\varphi_{j}$ is asymptotically $\alpha(D_{j})^{1/(2j)}$-symmetric
and $\lim_{j}\varphi_{j}=\varphi$ uniformly. By Corollary \ref{cor:amenable-fullpressure-bip},
we have $\mathcal{P}\left(\varphi_{j}\circ\pi_{1},\sigma\rtimes\Psi\right)\ge\mathcal{P}\left(\varphi_{j},\sigma\right)-\log\big(\alpha(D_{j})^{1/(2j)}\big)$.
The claim follows by letting $j$ tend to infinity, because we have
$\lim_{j}(D_{j})^{1/(2j)}=1$, $\lim_{j}\mathcal{P}\left(\varphi_{j}\circ\pi_{1},\sigma\rtimes\Psi\right)=\mathcal{P}\left(\varphi\circ\pi_{1},\sigma\rtimes\Psi\right)$
and $\lim_{j}\mathcal{P}\left(\varphi_{j},\sigma\right)=\mathcal{P}\left(\varphi,\sigma\right)$. 

Finally, suppose that $\varphi$ is compactly asymptotically $\alpha$-symmetric
with respect to $\Psi$. For each $k\in\N$, define\foreignlanguage{english}{
$\varphi_{k}:=\varphi_{|\Sigma_{k}}$. Since $G_{k}:=\Psi\left(\Sigma_{k}^{*}\right)$
is a subgroup of the amenable group $G$, we have that $G_{k}$ is
amenable. Moreover,  $\varphi_{k}$ is }asymptotically\foreignlanguage{english}{
$\alpha$-symmetric with respect to $\Psi_{|I_{k}^{*}}$. Hence, by
the first part of the proof, we have $\mathcal{P}\left(\varphi_{k}\circ\pi_{1},\sigma\rtimes\Psi\right)\ge\mathcal{P}\left(\varphi_{k},\sigma\right)-\log\alpha$.
The result follows from the facts that $\lim_{k}\mathcal{P}\left(\varphi_{k}\circ\pi_{1},\sigma\rtimes\Psi\right)=\mathcal{P}\left(\varphi\circ\pi_{1},\sigma\rtimes\Psi\right)$
and $\lim_{k}\mathcal{P}\left(\varphi_{k},\sigma\right)=\mathcal{P}\left(\varphi,\sigma\right)$.
This was proved for} Hölder continuous functions in \cite[Theorem 2]{MR1738951},
and it is straightforward to extend the proof to functions of medium
variation.
\end{proof}
\selectlanguage{american}%
\providecommand{\bysame}{\leavevmode\hbox to3em{\hrulefill}\thinspace}
\providecommand{\MR}{\relax\ifhmode\unskip\space\fi MR }
\providecommand{\MRhref}[2]{%
  \href{http://www.ams.org/mathscinet-getitem?mr=#1}{#2}
}
\providecommand{\href}[2]{#2}

\end{document}